\documentclass[journal]{IEEEtran}

\usepackage{graphicx}
\usepackage{amsmath,amsthm}
\usepackage{cite}
\usepackage{amssymb}
\usepackage{subfigure}
\usepackage{subfigmat}
\usepackage{epstopdf}

\usepackage{algorithm}
\usepackage{algorithmic}

\usepackage{multirow,multicol,threeparttable}
\usepackage{booktabs}
\usepackage{enumerate}

\def\BibTeX{{\rm B\kern-.05em{\sc i\kern-.025em b}\kern-.08em T\kern-.1667em\lower.7ex\hbox{E}\kern-.125emX}}

\newtheorem{theorem}{Theorem}

\newtheorem{lemma}{Lemma}
\newtheorem{fact}{Fact}

\begin{document}

\title{\LARGE \bf Shortest Paths of Bounded Curvature for the Dubins Interval Problem}

\author{Satyanarayana Manyam$^{1}$, Sivakumar Rathinam$^{2}$, David Casbeer$^{3}$ and Eloy Garcia$^{4}$%
\thanks{$^{1}$National Research Council Fellow, Air Force Research Laboratory, Dayton-OH, 45433,}%
\thanks{$^{2}$Assistant Professor, Mechanical Engineering, Texas A \& M University, College Station, TX-77843, }%
\thanks{$^{3}$Research Scientist, Air Force Research Laboratory, Dayton-OH, 45433,}%
\thanks{$^{4}$Research Scientist, Infoscitex Corp., Dayton-Ohio, 45431.}%
}
\markboth{}
{Murray and Balemi: Using the Document Class IEEEtran.cls} 


\maketitle
\thispagestyle{empty}
\pagestyle{empty}

\begin{abstract}
The Dubins interval problem aims to find the shortest path of bounded curvature between two targets such that the departure angle from the first target and the arrival angle at the second target are constrained to two respective intervals.  We propose a new and a simple algorithm to this problem based on the minimum principle of Pontryagin.
\end{abstract}

\section{Introduction}

Path planning problems involving Dubins vehicles have received significant attention in the literature due to their applications involving unmanned vehicles\cite{LeeDubins,Medeiros2010,Orient2014,macharet2013efficient,macharet2012data,sujit2013route,kenefic2008finding,macharet2011nonholonomic,Ozguner2005,ketan_2008,Rathinam_2007_IEEETASE,le2012dubins,ma2006receding}. A Dubins vehicle \cite{Dubins1957} is a vehicle that travels at a  constant speed and has a lower bound on the radius of curvature at any point along its path. The basic problem of finding a shortest path for a vehicle from a point  at $(x_1,y_1)$ with heading $\theta_1$ to a point at $(x_2,y_2)$ with heading $\theta_2$ was solved by Dubins in \cite{Dubins1957}, and later by authors in \cite{boissonat,bui} using Pontryagin's minimum principle\cite{pontryagin}. This article considers a generalization of this standard problem called the Dubins Interval Problem and is stated as follows: Given two targets located at $(x_1,y_1)$ and $(x_2,y_2)$, respectively, on a plane, a closed interval $\Theta_1$ of departure angles from target 1, and a closed interval $\Theta_2$ of arrival angles at target 2, find a departure angle $\theta_1 \in \Theta_1$, an arrival angle $\theta_2 \in \Theta_2$ and a path from $(x_1,y_1,\theta_1)$ to $(x_2,y_2,\theta_2)$ such that the radius of curvature at any point in the path is lower bounded by $\rho$ and the length of the path is a minimum (refer to Fig. \ref{fig:dI}).

\begin{figure}[h!]
\centering{}
\includegraphics[width=3in]{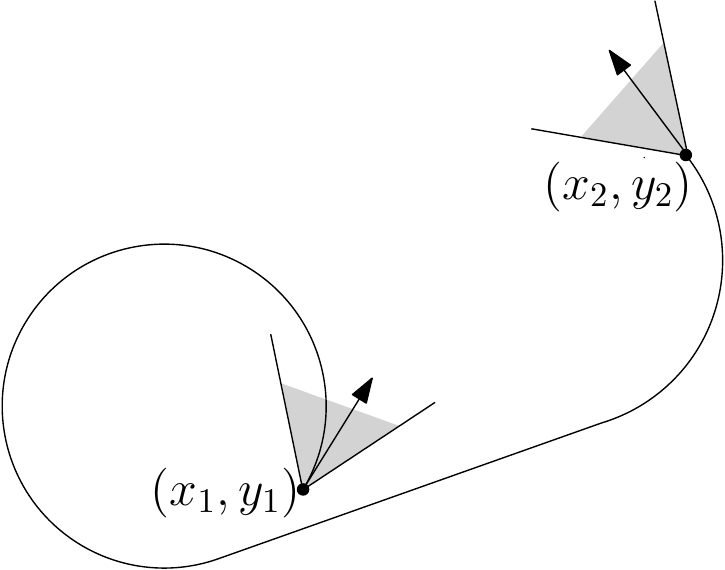}
\caption{A feasible solution to the Dubins Interval Problem.}
\label{fig:dI}
\end{figure}

Variants of the Dubins interval problem arise in search and attack problems\cite{garcia2014cooperative,garcia2015active} where a vehicle must reach a target such that the arrival angle of the vehicle at the target must be within given bounds. The Dubins interval problem also arises while lower bounding Traveling Salesman Problems (TSPs) involving Dubins vehicles\cite{Rathinam_lb2015}. In \cite{Rathinam_lb2015}, the lower bounding problem was posed as a generalized TSP where the cost of traveling between any two nodes requires one to solve the Dubins interval problem. The Dubins interval problem was solved using calculus and some monotonicity properties of the optimal paths in \cite{Rathinam_lb2015}. In this article, we give a simple and a direct algorithm using Pontryagin's minimum principle\cite{pontryagin}. The application of the minimum principle leads to an additional set of complementary slackness conditions corresponding to the angle constraints at the targets. The key contribution of this article lies in interpreting these conditions and characterizing an optimal solution to the Dubins interval problem. Similar to solving the standard Dubins problem, an optimal solution to the Dubins interval problem can be obtained by simply comparing the length of few candidate solutions. After proving the main result for the Dubins interval problem, we also solve the special case of the Dubins interval problem where the departure angle is given at target 1 while a closed interval of angles is given at target 2. In this special case, the objective is to find an optimal Dubins path and the corresponding arrival angle at target 2.

\section{Notations}
 The two targets lie on a ground plane and the motion of the vehicle also occurs on the same plane. As commonly assumed, any angle measured with respect to the $x$-axis in the counterclockwise direction is positive. The interval $\Theta_k$ at target $k$ is defined as $\Theta_k = [\theta^{min}_k,\theta^{max}_k]\subseteq [0,2\pi]$ with $\theta^{min}_k<\theta^{max}_k$ for $k=1,2$. Given an initial configuration $(x_{1},y_{1},\theta_{1})$ and a final configuration $(x_{2},y_{2},\theta_{2})$, L.E. Dubins \cite{Dubins1957} showed that the shortest path for a vehicle to travel between the two configurations subject to the minimum turning radius ($\rho$) constraint must consist of at most three segments where each segment is a circle of radius $\rho$ or a straight line. In particular, if a curved segment of radius $\rho$ along which the vehicle travels in a counterclockwise (clockwise) rotational motion is denoted by $L(R)$, and the segment along which the vehicle travels straight is denoted by $S$, then the shortest path is one of $RSR$, $RSL$, $LSR$, $LSL$, $RLR$ and $LRL$ or a degenerate form of these paths. For example, the degenerate forms of $RSL$ are $S$, $L$, $R$, $RS$, $SL$ and $RL$. We also subscript a curved segment in some places ($L_\psi$ or $R_\psi$) to indicate the angle of turn ($\psi$) in the curved segment. In the results stated in the ensuing sections, the open interval $(\theta_k^{min}, \theta_k^{max})$ is denoted as $\Theta_k^o$ for $k=1,2$.

\section{Main Result}

\begin{theorem}\label{theorem:main}
Any shortest path which is $C^1$ and piecewise $C^2$ of bounded curvature between the two targets with the departure angle $\theta_d \in \Theta_1 =[\theta_1^{min},\theta_1^{max}]$
 at target 1 and the arrival angle $\theta_a\in \Theta_2 = [\theta_2^{min},\theta_2^{max}]$
 at target 2 must be one of the following or a degenerate form of these: \\
\begin{enumerate}[{Case} 1:]
\item $S$ or $L_\psi$ or $R_\psi$ or $L_\psi R_\psi$ or $R_\psi L_\psi$ with $\psi> \pi$. \\
\item $\theta_d=\theta_1^{max}$ and $\theta_a=\theta_2^{max}$ and the path is $LSR$.
\item $\theta_d=\theta_1^{max}$ and $\theta_a=\theta_2^{min}$ and the path is either $LSL$ or $LR_\psi L$ with $\psi>\pi$.
\item $\theta_d=\theta_1^{min}$ and $\theta_a=\theta_2^{min}$ and the path is $RSL$.
\item $\theta_d=\theta_1^{min}$ and $\theta_a=\theta_2^{max}$ and the path is either $RSR$ or $RL_\psi R$ with $\psi>\pi$.\\

\item $\theta_d=\theta_1^{max}$ and $\theta_a \in \Theta_2^o$ and the path is either $LS$ or $LR_\psi $ with $\psi >\pi$.
\item $\theta_d=\theta_1^{min}$ and $\theta_a \in \Theta_2^o$ and the path is either $RS$ or $RL_\psi$ with $\psi >\pi$. \\
\item $\theta_d \in \Theta_1^o$ and $\theta_a = \theta_2^{max}$ and the path is either $SR$ or $L_\psi R$ with $\psi >\pi$.
\item $\theta_d \in \Theta_1^o$ and $\theta_a = \theta_2^{min}$ and the path is either $SL$ or $R_\psi L$ with $\psi >\pi$. \\
\end{enumerate}
\end{theorem}
\textbf{Remark:} Note that in cases 1 and 6-9, the departure and arrival angles are implicitly specified by each of the paths. For example, if $L_\psi R_\psi$ with $\psi>\pi$ exists between the two targets, as the length of the segment $L$ is equal to the length of the segment $R$, the departure and arrival angles are simply specified by geometry (we will later discuss this in the proofs; refer to Fig.\ref{fig:CC}). Similarly, if $\theta_d=\theta_1^{max}$ (case 6), the arrival angle at target 2 is determined by the $LS$ or $LR$ paths. The only remaining part would be to check if the arrival angle at target 2 lies in the interval $[\theta_2^{min},\theta_2^{max}]$. If it does, then the corresponding path is a candidate for an optimal solution to the Dubins interval problem.

\section{The minimum principle}
Let $v_o$ be the speed of the vehicle, and $u(t)$ denote the control input for the vehicle at time $t$. Let $x(t),y(t),\theta(t)$ denote the position and angle coordinates of the vehicle as a function of time on the ground plane. At time $t=0$, let the vehicle be located at $(x_1,y_1)$. Let the travel time of the vehicle be denoted by $T$. Note that $\theta_d = \theta(0)$ and $\theta_a = \theta(T)$. The Dubins interval problem can be formulated as an optimal control problem as follows:

\begin{align}
\min_{u(t)\in [-1,1]} &\int_0^T 1 dt
\end{align}
subject to
\begin{align}
\frac{dx}{dt} &= v_o \cos \theta, \nonumber \\
\frac{dy}{dt} &= v_o \sin \theta, \nonumber \\
\frac{d \theta}{dt} &= \frac{u}{\rho},  \label{eq:kinematics}
\end{align}
and the following boundary conditions:
\begin{align}
x(0)=x_1, &~  x(T)=x_2, \\
y(0)=y_1, &~  y(T)=y_2,
\end{align}

\begin{align}
\theta_{1}^{min}-\theta(0)\leq 0, \label{eq:theta0_1}\\
\theta(0) - \theta_{1}^{max} \leq 0,\\
\theta_{2}^{min}-\theta(T)\leq 0, \\
\theta(T) - \theta_{2}^{max} \leq 0. \label{eq:thetaT_2}
\end{align}

Let the adjoint variables associated with $p(t)=(x(t),y(t),\theta(t))$ be denoted as $\Lambda(t)=$ $(\lambda_x(t),\lambda_y(t),\lambda_{\theta}(t))$. The Hamiltonian associated with above system is defined as:
\begin{align}
H(\Lambda,p,u) = 1 + v_o \cos \theta \lambda_x + v_o \sin \theta \lambda_y + \frac{u}{\rho} \lambda_{\theta},
\end{align}
and the differential equations governing the adjoint variables are defined as:
\begin{align}
\frac{d\lambda_x}{dt} &= 0, \nonumber \\
\frac{d\lambda_y}{dt} &= 0, \nonumber \\
\frac{d \lambda_{\theta}}{dt} &= v_o \sin \theta \lambda_x - v_o \cos \theta \lambda_y. \label{eq:lambda}
\end{align}

Applying the fundamental theorem of Pontryagin\cite{pontryagin} to the above problem, we obtain the following: If $u^*$ is an optimal control to the Dubins interval problem, then there exists a non-zero adjoint vector $\Lambda(t)$ and $T>0$ such that $p(t),\Lambda(t)$ being the solution to the equations in \eqref{eq:kinematics} and \eqref{eq:lambda} for $u(t) = u^*(t)$, the following conditions must be satisfied: \\

\begin{itemize}
\item  $\forall t\in [0,T]$, $H(\Lambda,p,u^*) \equiv \min_{u\in [-1,1]}H(\Lambda,p,u)$.
\item $\forall t\in [0,T]$, $H(\Lambda,p,u^*) \equiv 0$.
\item Suppose $\alpha_1,\alpha_2,\beta_1,\beta_2$ are the Lagrange multipliers corresponding to the boundary conditions in \eqref{eq:theta0_1}-\eqref{eq:thetaT_2} respectively. Then, we have,
    \begin{align}
    \alpha_1,\alpha_2,\beta_1,\beta_2  \geq 0, \label{eq:dual}\\
    \alpha_1(\theta_{1}^{min}-\theta(0)) =  0, \label{eq:alpha1} \\
\alpha_2(\theta(0) - \theta_{1}^{max}) = 0, \label{eq:alpha2}\\
\beta_1(\theta_{2}^{min}-\theta(T)) = 0, \label{eq:beta1}\\
\beta_2(\theta(T) - \theta_{2}^{max}) = 0, \label{eq:beta2}\\
\lambda_{\theta}(T) = \beta_2-\beta_1, \label{eq:lambdaT}\\
\lambda_{\theta}(0) = \alpha_1-\alpha_2. \label{eq:lambda0}
    \end{align}
\end{itemize}

Given a departure angle at target 1 and an arrival angle at target 2, the following facts are known for the basic Dubins problem in \cite{boissonat},\cite{bui}. We will use them in our proofs later. \\
\begin{fact}\label{fact1}
Consider any point $P$ on an optimal path which is either an inflexion point of the path (point joining two curved segments or a point joining a curved segment and a straight line) or any point on a straight line segment of the path. Suppose the vehicle crosses this point at time $t\in[0,T]$. Then, $\lambda_{\theta}(t)=0$. \\
\end{fact}

\begin{fact}\label{fact2}
All the points of an optimal path where $\lambda_{\theta}(t)=0$ lie on the same straight line. \\
\end{fact}

\begin{fact}\label{fact3}
Consider any curved segment of an optimal path. Let the times $t_1,t_2$ be such that $0\leq t_1<t_2\leq T$, the vehicle is located on the curved segment at times $t_1,t_2$, and $\lambda_{\theta}(t_1)=\lambda_\theta(t_2)=0$. Then, the length of the curved segment between the times $t_1$ and $t_2$ must be greater than $\pi\rho$. \\
\end{fact}

\begin{fact}\label{fact4}
For any $t\in [0,T]$, the optimal control $u^*(t)=-sign(\lambda_\theta(t))$ if $\lambda_\theta(t)\neq 0$.  \\
\end{fact}

\section{Proof of theorem \ref{theorem:main}}


%
%
%
%

The departure angle in any optimal solution must either be an interior point in $\Theta_1$ or belong to one of the boundary values of $\Theta_1$. Similarly, the arrival angle in any optimal solution must either be an interior point in $\Theta_2$ or belong to one of the boundary values of $\Theta_2$. The combination of choices for the departure angle and arrival angle coupled with the complementary slackness equations in \eqref{eq:dual}-\eqref{eq:lambda0} provides constraints for the values of $\lambda_\theta(0)$ and $\lambda_{\theta}(T)$. These constraints can then be used to find the candidate solutions for solving the Dubins interval problem. The following lemma first shows the relationship between the departure, arrival angles and the constraints on $\lambda_\theta(0)$ and $\lambda_{\theta}(T)$. \\

\begin{lemma}\label{lemma:table1}
The constraints for the adjoint variable $\lambda_\theta$ at times $t=0$ and $t=T$ corresponding to the departure and arrival angles are given in the table below: \\

{{\small
\begin{center}
\begin{tabular}{p{.2in}p{1.5in}p{1.2in}}
\hline
Case No. & Conditions on the departure and arrival angles of any optimal path & Implied constraints on $\lambda_{\theta}(0)$ and $\lambda_{\theta}(T)$  \\ \hline
1 & $\theta(0) \in \Theta_1^o$, $\theta(T)\in \Theta_2^o$ & $\lambda_ {\theta}(0)=0,\lambda_ {\theta}(T)=0$ \\
2 & $\theta(0) = \theta_1^{max} $, $\theta(T)=\theta_2^{max}$ & $\lambda_{\theta}(0) \le 0,\lambda_{\theta}(T) \ge 0$   \\
3 & $\theta(0) = \theta_1^{max} $, $\theta(T)=\theta_2^{min}$ & $\lambda_{\theta}(0) \le 0,\lambda_{\theta}(T) \le 0$  \\
4 & $\theta(0) = \theta_1^{min} $, $\theta(T)=\theta_2^{min}$ & $\lambda_{\theta}(0) \ge 0,\lambda_{\theta}(T) \le 0$  \\
5 & $\theta(0) = \theta_1^{min} $, $\theta(T)=\theta_2^{max}$ & $\lambda_{\theta}(0) \ge 0,\lambda_{\theta}(T) \ge 0$ \\
6 & $\theta(0) = \theta_1^{max} $, $\theta(T)\in \Theta_2^o$ &  $\lambda_{\theta}(0) \le 0,\lambda_ {\theta}(T)=0$ \\
7 & $\theta(0) = \theta_1^{min} $, $\theta(T)\in \Theta_2^o$ &  $\lambda_{\theta}(0) \ge 0,\lambda_ {\theta}(T)=0$ \\
8 & $\theta(0) \in \Theta_1^o$, $\theta(T) = \theta_2^{max}$ &  $\lambda_{\theta}(0) =0,\lambda_ {\theta}(T)\ge 0$  \\
9 & $\theta(0) \in \Theta_1^o$, $\theta(T) = \theta_2^{min}$ &  $\lambda_{\theta}(0) =0,\lambda_ {\theta}(T) \le 0$ \\ \hline  \\
\end{tabular}
\end{center}
}}

\end{lemma}

\begin{proof}
  Consider the first set of conditions, $\theta(0) \in \Theta_1^o$, $\theta(T)\in \Theta_2^o$, in the table. If $\theta(0) \in \Theta_1^o$, then equations \eqref{eq:alpha1} and \eqref{eq:alpha2} imply $\alpha_1=0$ and $\alpha_2=0$ respectively. Using \eqref{eq:lambda0}, we obtain $\lambda_\theta(0)=\alpha_1-\alpha_2=0$. Similarly, if $\theta(T)\in \Theta_2^o$, then equations \eqref{eq:beta1} and \eqref{eq:beta2} imply $\beta_1=0$ and $\beta_2=0$ respectively. Using \eqref{eq:lambdaT}, we obtain $\lambda_\theta(T)=\beta_2-\beta_1=0$. Each of the other constraints in the table can be verified using a similar procedure. Hence proved. \\
\end{proof}

We now use the constraints on $\lambda_{\theta}(0)$ and $\lambda_{\theta}(T)$ to infer the candidate solutions for the Dubins interval problem. \\

\begin{lemma}\label{lemma:00}
Let $\lambda_{\theta}(0)=\lambda_{\theta}(T) = 0$. Then the optimal path for the Dubins interval problem must be either $S$ or $L_\psi$ or $R_\psi$ or $L_\psi R_\psi$ or $R_\psi L_\psi$ with $\psi>\pi$. \\
\end{lemma}

\begin{proof}
An optimal path can just be a straight line from fact \ref{fact2}. From fact \ref{fact3}, a curved segment of length greater than $\pi\rho$ can satisfy the boundary conditions $\lambda_{\theta}(0)=\lambda_{\theta}(T) = 0$. In addition, any path containing three curved segments and satisfying the boundary conditions must have the length of each curved segment (between any two inflexion points) greater than $\pi\rho$; however, as shown in \cite{bui}, such a path cannot be optimal. Therefore, an optimal path may consist of either one or two curved segments with the length of each segment greater than $\pi\rho$. \\

If an optimal path consists of exactly two curved segments, then there are three points ($(x_1,y_1)$, $(x_2,y_2)$ and the inflexion point) where $\lambda_{\theta}(t)$ becomes zero for $t\in [0,T]$. From fact \ref{fact2}, all these three points must lie on the same straight line. Therefore, the length of the first curved segment must be equal to the length of the second curved segment as shown in Fig. \ref{fig:CC} (in this case, $\theta(0)=\theta(T)$). \\
\end{proof}

\begin{figure}[htb]
\centering{}
\includegraphics[width=3in]{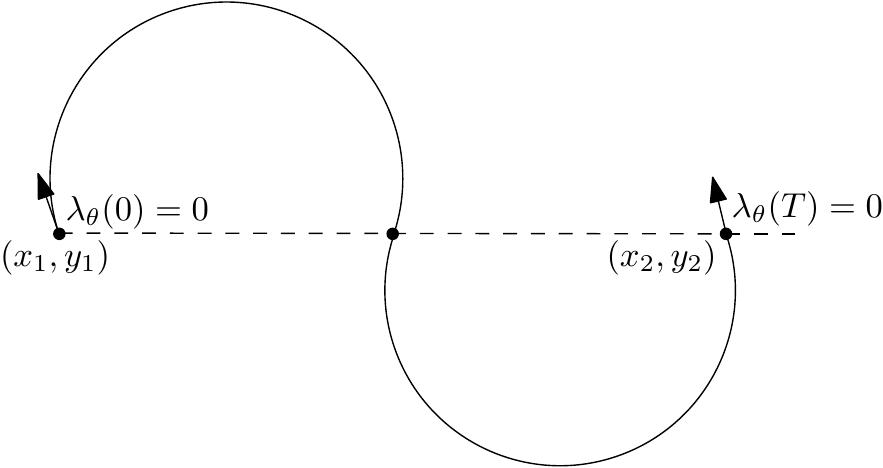}
\caption{A $RS$ path with the boundary values of $\lambda_\theta$ equal to 0.}
\label{fig:CC}
\end{figure}

Note that a set of constraints, say $\lambda_{\theta}(0) \le 0, \lambda_{\theta}(T) \ge 0$, can be satisfied if either $\lambda_{\theta}(0) =\lambda_{\theta}(T) =0$ or $\lambda_{\theta}(0) < 0, \lambda_{\theta}(T) = 0$ or $\lambda_{\theta}(0) < 0, \lambda_{\theta}(T) > 0$ or $\lambda_{\theta}(0) = 0, \lambda_{\theta}(T) > 0$ is satisfied. The following lemma identifies all the candidate solutions for any combination of these constraints on $\lambda_{\theta}(0)$ and  $\lambda_{\theta}(T)$. \\

\begin{lemma}\label{lemma:table2}
Candidate optimal paths for the Dubins interval problem corresponding to the constraints on $\lambda_{\theta}(0)$ and $\lambda_{\theta}(T)$ are given in the table below: \\

{{\small
\begin{center}
\begin{tabular}{p{1.2in}p{1.5in}}
\hline
Constraints on $\lambda_{\theta}(0)$ and $\lambda_{\theta}(T)$ & Candidate solutions \newline $(\psi>\pi)$  \\ \hline
$\lambda_ {\theta}(0)=0,\lambda_ {\theta}(T)=0$ & $S$,$L_\psi$,$R_\psi$,$L_\psi R_\psi$,$R_\psi L_\psi$\\
$\lambda_{\theta}(0) < 0,\lambda_{\theta}(T) = 0$ & $LS, LR_\psi$ with $\psi>\pi$\\
$\lambda_{\theta}(0) < 0,\lambda_{\theta}(T) < 0$ & $LSL, LR_\psi L$ with $\psi>\pi$\\
$\lambda_{\theta}(0) < 0,\lambda_{\theta}(T) > 0$  & $LSR$ \\
$\lambda_{\theta}(0) = 0,\lambda_{\theta}(T) < 0$ & $SL, RL_\psi$ with $\psi>\pi$ \\
$\lambda_{\theta}(0) = 0,\lambda_ {\theta}(T) > 0$ & $SR, LR_\psi$ with $\psi>\pi$\\
$\lambda_{\theta}(0) > 0,\lambda_{\theta}(T) = 0$  & $RS, RL_\psi$ with $\psi>\pi$ \\
$\lambda_{\theta}(0) > 0,\lambda_ {\theta}(T)<0$ & $RSL$\\
$\lambda_{\theta}(0) > 0,\lambda_{\theta}(T) > 0$ & $RSR, RL_\psi R$ with $\psi>\pi$\\
\hline  \\
\end{tabular}
\end{center}
}}
\end{lemma}

\begin{proof}
The candidate solutions corresponding to $\lambda_\theta(0)=\lambda_\theta(T)=0$ has already been shown in lemma \ref{lemma:00}. Consider the next set of constraints $\lambda_\theta(0)< 0$ and $\lambda_\theta(T) = 0$. If $\lambda_\theta(0)< 0$, the first segment of the path must be $L$ (fact \ref{fact4}). $LSR_\psi$ or $LSL_\psi$ with $\psi>0$ is not possible because this path would violate fact \ref{fact2} unless the length of the straight line is equal to 0. $LRL$ is also not possible because the length of the $R$ segment and the last $L$ segment would each be greater than $\pi\rho$ which then cannot be optimal\cite{bui}. Therefore, the possible candidates are $LS$ or $LR_\psi$ with $\psi >\pi$. Each of the remaining set of the constraints can be shown using a similar procedure. \\
\end{proof}

The theorem follows by combining lemmas \ref{lemma:table1} and \ref{lemma:table2}.

\section{Solution for the Special Case}

In this special case of the Dubins interval problem, the departure angle $\theta_d$ is given while the arrival angle at target 2 must belong to a closed interval $\Theta_2$.

\begin{theorem}\label{theorem:main2}
Any shortest path which is $C^1$ and piecewise $C^2$ of bounded curvature between the two targets with the departure angle $\theta_d$
 at target 1 and the arrival angle $\theta_a\in \Theta_2 = [\theta_2^{min},\theta_2^{max}]$ at target 2 must be one of the following or a degenerate form of these: \\
\begin{enumerate}[{Case} 1:]
\item The path is either $LS$ or $RS$ or $LR_\psi $ or $RL_\psi $ with $\psi >\pi$.

\item $\theta_a=\theta_2^{max}$ and the path is either $LSR$ or $RSR$ $RL_\psi R$ with $\psi >\pi$.
\item $\theta_a=\theta_2^{min}$ and the path is either $RSL$ or $LSL$ $LR_\psi L$ with $\psi >\pi$.
\end{enumerate}
\end{theorem}

This theorem can be proved following the same procedure as in the previous section.

\bibliographystyle{IEEEtran}
\bibliography{NSFbib}

\end{document}